\documentclass[a4paper,11pt,reqno]{amsart}

\usepackage[utf8]{inputenc}
\usepackage[T1]{fontenc}
\usepackage{lmodern}
\usepackage[english]{babel}
\usepackage{microtype}

\usepackage{amsmath,amssymb,amsfonts,amsthm,esint}
\usepackage{mathtools,accents}
\usepackage{mathrsfs}
\usepackage{aliascnt}
\usepackage{braket}
\usepackage{bm}

\usepackage[a4paper,margin=3.5cm]{geometry}
\usepackage[citecolor=blue,colorlinks]{hyperref}

\usepackage{enumerate}
\usepackage{xcolor}

\allowdisplaybreaks
\makeatletter
\g@addto@macro\@floatboxreset\centering
\makeatother

\makeatletter
\def\newaliasedtheorem#1[#2]#3{
	\newaliascnt{#1@alt}{#2}
	\newtheorem{#1}[#1@alt]{#3}
	\expandafter\newcommand\csname #1@altname\endcsname{#3}
}
\makeatother

\numberwithin{equation}{section}

\newtheoremstyle{ ed}{\topsep}{\topsep}{\slshape}{}{\bfseries}{.}{.5em}{}

\theoremstyle{plain}
\newtheorem{theorem}{Theorem}[section]
\newaliasedtheorem{proposition}[theorem]{Proposition}
\newaliasedtheorem{lemma}[theorem]{Lemma}
\newaliasedtheorem{corollary}[theorem]{Corollary}
\newaliasedtheorem{counterexample}[theorem]{Counterexample}
\newaliasedtheorem{fact}[theorem]{Fact}

\theoremstyle{definition}
\newaliasedtheorem{definition}[theorem]{Definition}
\newaliasedtheorem{question}[theorem]{Question}
\newaliasedtheorem{openquestion}[theorem]{Open Question}
\newaliasedtheorem{conjecture}[theorem]{Conjecture}

\theoremstyle{remark}
\newaliasedtheorem{remark}[theorem]{Remark}
\newaliasedtheorem{example}[theorem]{Example}

\newcommand{\setR}{\mathbb{R}}

\let\altphi\phi
\let\phi\varphi
\let\varphi\altphi
\let\altphi\undefined

\newcommand{\abs}[1]{\left\lvert#1\right\rvert}

\newcommand{\ip}[2]{\langle{#1},{#2}\rangle}

\newcommand{\meas}{\mathfrak{m}}

\DeclareMathOperator{\RCD}{RCD}

\newfont{\tmpf}{cmsy10 scaled 2500}

\def\XXint#1#2#3{{\setbox0=\hbox{$#1{#2#3}{\int}$ }
		\vcenter{\hbox{$#2#3$ }}\kern-.6\wd0}}

\begin{document}

	\title[Weyl's Lemma on $RCD(K,N)$ Spaces]{Weyl's Lemma on $RCD(K,N)$ Metric Measure Spaces}

	\author{Yu Peng}
	\address{Department of Mathematics, Sun Yat-Sen University, Guangzhou 510275, China,\ pengy86@mail2.sysu.edu.cn} 
	\author{Hui-Chun Zhang}
	\address{Department of Mathematics, Sun Yat-Sen University, Guangzhou 510275, China,\ zhanghc3@mail.sysu.edu.cn}
	\author{Xi-Ping Zhu}
	\address{Department of Mathematics, Sun Yat-Sen University, Guangzhou 510275, China,\ stszxp@mail.sysu.edu.cn}

	\maketitle
	
	\begin{abstract}
		In this paper, we extend the classical Weyl's lemma to   $RCD(K,N)$ metric measure spaces.    As  its applications,  we show  the local regularity of solutions for Poisson  equations  and a Liouville-type result  for $L^1$ very weak harmonic functions  on $RCD(K,N)$  spaces. Meanwhile, a byproduct is that we obtain a gradient estimate for solutions to a class of elliptic equations with dis-continuous coefficients.
	\end{abstract}
	
	\section{Introduction}
	The classical Weyl's lemma  states that any very weakly harmonic function in the Euclidean spaces $\setR^n$ must be smooth. Namely,
	\begin{theorem}[Weyl \cite{Weyl40}]
		Let $\Omega\subset\setR^n$ be open. Suppose that $u\in L_{\rm loc}^1(\Omega)$ and 
		\begin{equation*}
			\int_{\Omega}u\Delta \phi=0
		\end{equation*}
		for any $\phi\in C^{\infty}_c(\Omega)$. Then $u$ is  smooth (I.e. there exists a $C^\infty$-representative).
	\end{theorem}
	
	Extensions of Weyl's lemma are widely studied. For examples, Burch  \cite{Bur78} extended it to elliptic operators with  variable coefficients, H\"ormander \cite{H67} proved Weyl's lemma for hypoelliptic operators, and  Elson \cite{E74} developed a theory of distributions on separable real Banach spaces and proved an analogy of Weyl's lemma.   We refer to  \cite{Strook08} for a survey.
		 Recently, Di Fratta-Fiorenza \cite{DFF20} gave a short proof for the regularity of very weak solutions of Poisson equations via Weyl's lemma. 
	
	In this paper, we will extend Weyl's lemma to non-smooth settings. More precisely, we will  prove Weyl's lemma on $RCD(K,N)$  spaces. Given a metric measure space $(X,d,\meas)$, the  notion $RCD(K,N)$ is  a synthetic notion of lower Ricci bounds for  $(X,d,\meas)$, which  has been  developed in    \cite{Stu06a,Stu06b,LV09, AGS14a,AGS14b,Gig15,AMS16,EKS15}. The parameters $K\in\mathbb R$ and $N\in [1,+\infty]$ play the role of “Ricci curvature $\geqslant   K$ and dimension $\leqslant N$”.   We refer to the survey \cite{Amb18} for the improvements of geometric analysis on $RCD(K,N)$ spaces.
	Important examples of $RCD(K,N)$ spaces include Ricci limit spaces \cite{CC96, CC97, CC00} and finite-dimensional Alexandrov spaces with curvature bounded from below with respect to their Hausdorff measures \cite{Pet11, ZZ10}. 
	
	Let $(X,d,\meas)$ be an $RCD(K,N)$ space. The  Sobolev space $W^{1,2}(X)$ is a Hilbert space, and then  the inner product $\ip{\nabla u}{ \nabla v}\in L^1(X,\meas)$ makes sense for any $u,v\in W^{1,2}(X)$. Given an open subset $\Omega\subset X$, a function $u\in W^{1,2}_{\rm loc}(\Omega)$ is called a  (weakly) {\it harmonic function} on $\Omega$ if 
	\begin{equation*}
		\int_{X}\ip{\nabla u}{\nabla\phi }d\meas=0, \ \ \forall\phi\in Lip_c(\Omega),
	\end{equation*} 
	where $ Lip_c(\Omega)$ denotes the space of Lipschitz functions  defined on $X$ with compact supported in the interior of $\Omega$. The Lipschitz regularity of harmonic functions  on $RCD(K,N)$ spaces has been established in \cite{GigliMos15,Jia14,HKX13}.   
	
	In order to extend Weyl's lemma to the setting of $RCD(K,N)$ spaces, let us begin with the notion of ``very weakly harmonic function''.  Let $(X,d,\meas)$ be an $RCD(K,N)$ space. We denote by $P_t$  the heat flow  on $L^2(X,\meas)$ and  by $\Delta$  the associated infinitesimal generator with domain $D(\Delta)$.  
	Recall that the class of {\textit{test  functions}} on $RCD(K,N)$ spaces is given in \cite{GP20} by 
	\begin{equation*}
		\mathrm{Test}^\infty(X):=\left\{\phi\in D(\Delta)\cap L^\infty(X):\abs{\nabla \phi}\in L^\infty(X),\ \Delta \phi\in L^\infty(X)\cap W^{1,2}(X)\right\}.
	\end{equation*}	
	Given any open domain $\Omega\subset X$,  denote  by
	\begin{equation}\label{equ1.1}
	\mathrm{Test}_c^{\infty}(\Omega):=\left\{\phi\in  \mathrm{Test}^{\infty}(X) \big|  {\rm \ supp}(\phi)\subset \Omega\ {\rm and\ supp}(\phi)\ {\rm is\ compact}  \right\}.
	\end{equation}
It is well-known \cite{GP20} that the space $\mathrm{Test}^{\infty}(X)$ 	is dense in $W^{1,2}(X)$ and the space $\mathrm{Test}_c^{\infty}(\Omega)$ 	is dense in $W_0^{1,2}(\Omega).$	\begin{definition} 
		Let $(X,d,\meas)$ be an $RCD(K,N)$ space for some $1\leqslant N<\infty$ and $K\in \setR$ and let $\Omega\subset X$ be an open domain. A function  $u\in L_{\rm loc}^1(\Omega)$ is called a {\it very weakly subharmonic function} on $\Omega$ if 
	\begin{equation}\label{equ1.2}
	\int_{\Omega}u\Delta \phi d\meas\geqslant 0, \qquad \forall  \phi\in {\rm Test}^\infty_{c}(\Omega)\ \ {\rm and}\ \ \phi\geqslant 0.
	\end{equation}
	A function  $u\in L_{\rm loc}^1(\Omega)$ is called a {\it very weakly superharmonic function} on $\Omega$ if $-u$ is very weakly subharmonic. A function  $u\in L_{\rm loc}^1(\Omega)$ is called a {\it very weakly harmonic function} on $\Omega$ if both $u$ and $-u$ are very weakly subharmonic.
	\end{definition}
	
 The Weyl's lemma is concerned with the regularity of  a very weak harmonic function $u$ on $\Omega$. If  $u\in W^{1,2}_{\rm loc}(\Omega)$ then it is a (weakly) harmonic, and hence it is locally Lipschitz continuous (see  \cite{GigliMos15,Jia14,HKX13}).  This  remains the gap between the $u\in L^1_{\rm loc}(\Omega)$ and the  $u\in W^{1,2}_{\rm loc}(\Omega)$.
	
	The major obstacle of the improvement from $L^1_{\rm loc}(\Omega)$ to $W^{1,2}_{\rm loc}(\Omega)$ for very weak harmonic functions on a singular space  can be understood as follows.
For the convenience of the discussion, we consider a very weak harmonic function
  $u:(\Omega,g)\to \mathbb R$  from a domain $\Omega\!\subset\!\mathbb R^n$ with a \emph{singular} Riemannian metric $ g=(g_{ij})$. Then  $u\in L^1_{\rm loc}(\Omega)$ is a very weak solution of the  
elliptic equation of  divergence form
\begin{equation}\label{elliptic-equ}
\sum_{i,j=1}^n \partial_i\left(a^{ij} \partial_ju\right)=0\quad{\rm with}\ \   a^{ij}:=\sqrt gg^{ij},
  \end{equation}
in the sense of distributions, where $\partial_i=\frac{\partial }{\partial x_i}$,  $g=\det(g_{ij})$ and $(g^{ij})$ is the inverse matrix of $(g_{ij})$. It seems that Zhang-Bao's  work  \cite{ZB12} is the only  result fixing the whole gap from $L^1_{\rm loc}(\Omega)$ to $W^{1,2}_{\rm loc}(\Omega)$. They proved \cite{ZB12} that if the coefficients $ a^{ij}\in Lip_{\rm loc}(\Omega)$, then every $L^1_{\rm loc}(\Omega)$-solution of \eqref{elliptic-equ} must be in $W^{1,2}_{\rm loc}(\Omega)$.  Brezis \cite{Bre08} showed that if  all  $ a^{ij} $ are Dini continuous then every  $BV_{\rm loc}(\Omega)$-solution  of \eqref{elliptic-equ} must be in   $ W^{1,2}_{\rm loc}(\Omega)$. Recently, Manna-Leone-Schiattarella in \cite{MLS20} proved that if $ a^{ij}\in C^0(\Omega)\cap W^{1,n}_{\rm loc}(\Omega)$ satisfying a doubling Dini  condition, then every $   L^{\frac{n}{n-1}}_{\rm loc}(\Omega)$-solution of \eqref{elliptic-equ} is in  $  W^{1,2}_{\rm loc}(\Omega)$.  The
Dini-type condition of the coefficients $ a^{ij}$ is necessary:  Jin-Maz'ya-Schaftingen in \cite{JMS09}  constructed an equation \eqref{elliptic-equ} with $ a^{ij}\in C^0(\Omega)$ and that it has a solution $ u\in W^{1,1}_{\rm loc}(\Omega)\setminus W^{1,2}_{\rm loc}(\Omega)$. 
On the other hand,  we assume even that $\Omega$ is a small neighborhood near a regular point  in a finite dimensional Alexandrov space with curvature bounded from below.
 According to \cite{OS94, Per}, there is
 a  $BV_{\rm loc}$-Riemannian metric $(g_{ij})$ on $\Omega$.
However, it is well-known \cite{OS94} that these $g_{ij}$ may not be continuous on a \emph{dense} subset of $\Omega$. Thus, It seems the mentioned works \cite{JMS09, ZB12,Bre08} do not support that one can improve a $L^1_{\rm loc}(\Omega)$ very weak harmonic function to be in $W^{1,2}_{\rm loc}(\Omega)$, even on an Alexandrov space with curvature bounded from below.

	The  main result of this paper is the following Weyl's lemma on RCD-setting. 
	\begin{theorem}[Weyl's lemma]\label{weyl-lemma}
		Let $(X,d,\meas)$ be an $RCD(K,N)$ space for some $1\leqslant N<\infty$ and $K\in \setR$. Let $\Omega\subset X$ be an open domain. If  $u\in L_{\rm loc}^1(\Omega)$ is a very weakly harmonic function on $\Omega$,
		then $u$ is locally Lipschitz continuous on $\Omega$ (i.e. it has a locally Lipschitz continuous representative).	In particular, it is in $W^{1,2}_{\rm loc}(\Omega)$. 
	\end{theorem}

	\begin{remark}\label{remark1.4}
		When $X=\mathbb R^n$, the  regularity for very weak solutions of (\ref{elliptic-equ})  relies on the well-known Calder\'on-Zygmund theory which states  if $\Delta u\in L^p_{\rm loc}$ with $p>1$, then  $u\in W^{2,p}_{\rm loc}$. In particular, if $p>n$ then $|\nabla u|$ is in $C^\alpha_{\rm loc}$ for some $\alpha\in(0,1)$.
		Recently,  G. De Philippis and  J. N\'unez-Zimbr\'on \cite{DPNZ22} have shown  that this statement does not hold for general   Alexandrov spaces  with curvature bounded from below. For example, let us consider  the space $X^2$ constructed by  Ostu-Shioya in \cite{OS94}, which is a two-dimensional Alexandrov space  with the property that  the set  of singular points  is dense in $X^2$. 
Given any non-constant   harmonic function $u$ on a domain $\Omega\subset X^2$,  it was shown in \cite[Theorem 1]{DPNZ22} that  $\lim_{R\to0} \frac{1}{\meas(B_R(x))}\int_{B_R(x)}|\nabla u|^2(x)=0$ at every   point $x\in \Omega$ with ${\rm diam}(\Sigma_{x})<\pi$, where $\Sigma_x$ is the space of directions at $x$.  In the two-dimensional Alexandrov space $X^2$,   a point $x$ satisfying ${\rm diam}(\Sigma_x)<\pi$ is equivalent to that $x$ is  a singular point. Since the set of singular points on $X^2$ is dense, if $|\nabla u|$ were assumed to be continuous,  it would  follow $|\nabla u|\equiv0$ on $\Omega$. That is a contradiction.   
\end{remark}
					
		The first application of  Theorem \ref{weyl-lemma} is the following local regularity result for the  very weak solutions of Poisson equations on $RCD$ spaces. 
	\begin{corollary} \label{cor1.5}
		Let $(X,d,\meas)$ be an $RCD(K,N)$ space with $K\in\mathbb R$ and $N\in[1,\infty)$, and let $\Omega\subset X$ be an open subset. Suppose $f\in L^2_{\rm loc}(\Omega)$.  If   $u\in L_{\rm loc}^{1}(\Omega)$ is a very weak solution of Poisson equation on $\Omega$ in the sense of  
		\begin{equation}\label{equ1.5}
			\int_{\Omega} u\Delta\phi d\meas=\int_{\Omega}f\phi d\meas,\qquad \forall\phi\in \mathrm{Test}_c^{\infty}(\Omega)\ \ {\rm and}\ \ \phi\geqslant0.
		\end{equation}
		Then $u$ is in $W^{1,2}_{\rm loc}(\Omega)$. 
	\end{corollary}

	The second application of  Theorem \ref{weyl-lemma} is the following Liouville-type result  for $L^1$ very weakly subharmonic functions.	
	\begin{corollary}\label{cor1.6}
		Let $(X,d,\meas)$ be an $RCD(K,N)$ space  with $K\in\setR$ and $1\leq N<\infty$, then  any $L^1(X)$   very weakly subharmonic function on $X$ must be identically constant (namely, it has a constant representative).
	\end{corollary}

	Recall that Li-Schoen \cite{LS84} and Chung \cite{Chu83} gave some examples that a complete Riemannian manifold admits non-constant nonnegative $L^1$ harmonic functions.  Under an assumption of Ricci  lower bounds, P. Li \cite{PeterLi84} proved a   Liouville-type result for $L^1$ subharmonic functions   on complete Riemannian manifolds. We will prove  Corollary \ref{cor1.6} along the same line in \cite{PeterLi84} and combine Theorem \ref{weyl-lemma}.

At last, as a byproduct, we give a remark on the Lipschitz estimates for weak solutions  of a class of elliptic equations  on Euclidean spaces with {\it dis-continuous coefficients}. 
Let $n\geqslant 2$ and let $D\subset \mathbb R^n$ be a bounded domain. 
Suppose that $u\in W^{1,2}_{\rm loc}(D)$ is a weak solution of elliptic
equations of divergence form 
\begin{equation}\label{equation-1.5}
Lu:={\rm div}(A(x)Du)=\sum_{i,j=1}^n\partial_i( a^{ij}\partial_j u)=0,\quad 1\leqslant i, j\leqslant n,
\end{equation}
where the matrix of coefficients $A(x)=( a^{ij})_{i,j=1}^n$ is   symmetric, bounded, measurable and uniformly elliptic,  i.e.
\begin{equation}\label{equation-1.6}
\| a^{ij}\|_{L^\infty(D)}\leqslant \Lambda,\qquad   \lambda |\xi|^2\leqslant \sum_{i,j=1}^n a^{ij}(x)\xi_i\xi_j,\quad {\rm a.e.} \ x\in D, \quad \forall \xi\in \mathbb R^n
\end{equation}
for some $\lambda,\Lambda>0$.  It is well known \cite{Bur78}  that any weak solution $u$ of $Lu = 0$ is in $C^1(D)$  provided that $A$ satisfies the  Dini continuity condition. Recently, this Dini continuity condition were weaken to  others Dini-type conditions by Y. Li \cite{Li17}, Dong-Kim \cite{DK17} and Maz'ya-McOwen \cite{MMO11}.  However, an example constructed by Jin-Maz'ya-Schaftingen \cite{JMS09} shows that, general speaking,  the  continuity of the coefficient is not sufficient to ensure $u\in Lip_{\rm loc}(D)$.

Inspired by the regularity of harmonic functions on $RCD$ spaces, we first introduce   a class of models for elliptic coefficient, which plays the role of constant coefficient 
in the classical regularity theory.

\begin{definition}\label{cone-coefficient}
Let  $\overline{A}=(\overline{a}^{ij}): \mathbb R^n\to \mathbb R^n\times \mathbb R^n$ be   symmetric, bounded, measurable and uniformly elliptic on the each ball $B_R(0)\subset \mathbb R^n$ with $R>0$. Let $n\geqslant 3$. We first introduce a Riemannian metric  $\overline{g}_{ij}(x)$ on $\mathbb R^n$ by  
\begin{equation}\label{equation-1.7}
(\overline{g}_{ij})^{-1}(x)=\overline{g}^{ij}(x):=\frac{\overline{a}^{ij}(x)}{\left[{\rm det}(\overline{a}^{ij}(x))\right]^{\frac{1}{n-2}}}, \qquad  1\leqslant i,j\leqslant n.
\end{equation}
We denote by $d_{\overline{g}}$   the distance function induced by Riemannian metric $(\mathbb R^n, \overline{g}_{ij})$. 

The coefficient $\overline{A}$  is called a {\emph{conical coefficient of nonnegative curvature}}  if  the metric space $(\mathbb R^n,d_{\overline g})$ is a  cone with nonnegative curvature in the sense of Alexandrov. 
\end{definition}

We remark that if $\bar  a^{ij}$ is constant coefficient, then the Riemannian metric $\bar g_{ij}=const$, and then $(\mathbb R^n, \bar g_{ij})$ is flat (i.e., sectional curvature is zero). 

There are also some {\emph{ dis-continuous}} conical coefficient of nonnegative curvature. For example, consider the graph $\Gamma_f$ of a convex function $$f(x_1,x_2,\cdots, x_n)=\left(a_1x_1^2+a_2x^2_2+\cdots+a_nx^2_n\right)^{1/2},$$ where   $a_i\in (0,+\infty)$ for all $i=1,\cdots, n$. It is well known that $\Gamma_f$ is a cone of nonnegative curvature in the sense of Alexandrov. Under the coordinate system $(x_1,x_2,\cdots, x_n)$, its Riemannian metric is given by 
$$\bar g_{ij}=\delta_{ij}+\frac{a_ix_i}{f}\cdot \frac{a_jx_j}{f},$$
for all $1\leqslant i,j\leqslant n$. Then $\bar A(x)=(\bar a^{ij}(x))$ with 
$$\bar  a^{ij}:=\sqrt{{\rm det}(\bar g_{ij})} \bar g^{ij}=\left(1+\frac{\sum_{i=1}^n(a_ix_i)^2}{f^2} \right)^{1/2} \left(\delta_{ij}+\frac{a_ix_i}{f}\cdot \frac{a_jx_j}{f}\right)^{-1}$$
is a conical coefficient of nonnegative curvature, and $\bar  a^{ij}(x)$ is not continuous at $0.$

\begin{theorem}\label{theorem-1.8}
Let $u\in W^{1,2}_{\rm loc}(B_1(0))$ be a weak solution of elliptic equations of divergence form ${\rm div}(A(x)Du)=0,$
where the matrix of coefficients $A(x)=( a^{ij})_{i,j=1}^n$ is   symmetric, bounded, measurable and uniformly elliptic. Let $n\geqslant 3.$

 Suppose that   $A$ is Dini asymptotic to a conical coefficient of nonnegative
curvature in the sense that  there exists a conical coefficient of nonnegative curvature $\overline{A}=(\overline{a}^{ij})$ such that 
\begin{equation}\label{equ-dini}
\int_0^1\frac{\omega_{A,\overline{A}}(t)}{t}dt<\infty,\quad\   where \ \ \omega_{A,\overline{A}}(t):= \|A(x)-\overline{A}(x)\|_{L^\infty(B_t(0))}.
\end{equation}
Then 
   \begin{equation}\label{equation-1.9}
\limsup_{r\to0^+}   \fint_{B_{r}(0)}|\nabla u|^2d\mathcal L^n  \leqslant   C_0  \fint_{B_{1}(0)}|\nabla u|^2d\mathcal L^n,
\end{equation}
where the constant $C_0$ depends only on $n, \lambda, \Lambda$ and $ \omega_{A,\overline A}$, and $\mathcal L^n$ is the $n$-dimensional Lebesgue measure on $\mathbb R^n$.
\end{theorem}
This extends the result in \cite{Bur78}.

	\textbf{Organization of the paper.}  In Section \ref{SecPre}, we will recall some ingredients  on $RCD(K,N)$ spaces. In Section \ref{sect-weyl-lem}, we will prove  Weyl's lemma (Theorem \ref{weyl-lemma}).   Corollary \ref{cor1.5} and \ref{cor1.6} will be proved in section \ref{sect4} and \ref{sect5}.  Finally, we discuss the gradient estimates of solutions to elliptic equations of divergence form, Theorem \ref{theorem-1.8}, in section \ref{sec6}.
	
	\textbf{Acknowledgment.} The second author was partially supported by NSFC 12025109.
	The third author was partially supported by NSFC 12271530.

	\section{Preliminaries}\label{SecPre}

	Let $(X,d,\meas)$  be a metric measure space, i.e. $(X,d)$ is a complete and separable metric space endowed with a non-negative Borel metric which is finite on bounded sets. 
	Throughout this paper, we always   assume that $(X,d,\meas)$ is an $RCD(K,N)$ space with $N\in[1,+\infty)$ and $K\in\mathbb R$.  For explicit definitions, one can refer to  \cite{Gig15,AMS16,EKS15} and the survey \cite{Amb18}.
	
	Let $(X,d,\meas)$ be an $RCD(K,N)$ space with $K\in\mathbb R$ and $N\in[1,+\infty)$. The generalized Bishop-Gromov inequality (see \cite{LV09,Stu06a}) implies  a local measure doubling: for any $R>0$, there exists a constant $C_{N,K,R}>0$ such that 
	\begin{equation}\label{equ2.1}
		\frac{\meas(B_{r_2}(x))}{\meas(B_{r_1}(x))}\leq C_{N,K,R} \cdot \left(\frac{r_2}{r_1}\right)^N,\quad \forall r_1,r_2\in (0,R)\ \ {\rm with} \ \ r_1<r_2.
	\end{equation}

	The Sobolev spaces on $(X,d,\meas)$ were established in \cite{Che99,Shan00,AGS14a,Gig15} via several different approaches. We refer the readers to the textbook \cite[Chapter 2]{GP20} for basic definitions and properties of $W^{1,2}(X):=W^{1,2}(X,d,\meas)$ and $W^{1,2}_{\rm loc}(X):=W^{1,2}_{\rm loc}(X,d,\meas)$. For any $f\in W^{1,2}(X),$ we denote by $|\nabla f|\in L^2(X)$ the minimal weak upper gradient of $f$ (see for example \cite[Sect. 2.1]{GP20}). For any $f,h\in W^{1,2}(X)$, the inner product 
	$\ip{\nabla f}{\nabla h}$ is given by
	\begin{equation*} 
		\ip{\nabla f}{\nabla h}= \frac{1}{4}\left(|\nabla(f+h)|^2-|\nabla (f-h)|^2\right) \ \in L^1(X).
	\end{equation*}
The inner product $\ip{\nabla f}{\nabla h}$ is symmetric and bi-linear, and satisfies Cauchy-Schwarz inequality, Chain rule and  Leibniz rule (see \cite[Appendix]{Stu94} or  \cite{GP20}).
	
  For any open domain $\Omega\subset X$, the space $W^{1,2}_{0}(\Omega)$ is the closure of $ Lip_c(\Omega)$ with support in $\Omega$ under $W^{1,2}(X)$-norm.
 The space $W^{1,2}_{\rm loc}(\Omega)$ is the space of Borel functions $f:\Omega\to \mathbb R$ such that $f\chi\in W^{1,2}_{\rm loc}(X)$ for any Lipschitz function  $\chi:X\to [0,1]$ such that $d\big( {\rm supp}(\chi),X\setminus\Omega\big)>0$, where a function $f\chi$ is taken $0$ by definition on $X\setminus \Omega.$	
 
	Let $\{P_t\}_{t>0}$ be the heat flow on $X$ associated with the energy $\int_X|\nabla f|^2d\meas$ for any $f\in W^{1,2}(X)$, and let $\Delta $ be the infinitesimal generator of $P_t$ with domain 
	$D(\Delta)$. It is well know that (see for instance \cite[Chapter 5]{GP20})
	\begin{equation}\label{equ2.2}
	\int_Xh\Delta fd\meas=-\int_X\ip{\nabla f}{\nabla h}d\meas,\quad \ \forall f\in D(\Delta), \ \ h\in W^{1,2}(X).
	\end{equation}
		For any $f\in L^2(X)$, it is well known that $P_tf\in D(\Delta)$, and that the curve $t\mapsto P_t f$ is in $C^1\left((0,+\infty), L^{2}(X)\right)$ with 
	$\frac{\mathrm{d}}{\mathrm{d}t}P_t f= \Delta P_t f.$
	If $f\in D(\Delta)$, then 
	\begin{equation*}
		P_t(\Delta f)=\Delta(P_tf),\qquad \forall t>0.
	\end{equation*}
	Since $P_t$ is a family of linear contractions on $L^2(X)$ for any $t>0$, it can be uniquely extended  to be a family of linear contractions on $L^p(X)$ for any $p\in[1,+\infty]$.

	Since $RCD(K,N)$ spaces are locally doubling and    satisfy a local Poincar\'e inequality (see \cite{VR08, Ra12}), according to \cite{Stu94,Stu96,AGS14b}, there exists a  heat kernel  
	$p_t(x,y)$ on $(0,+\infty)\times X\times X$ such that 
	\begin{equation}
		P_t f(x)=\int_{X}p_t(x,y)f(y)d\meas(y)\ \ \ \ \forall\ t>0
		\nonumber
	\end{equation}
	for every $f\in L^2(X)$. It was proved \cite{Stu96,AGS14b} that $(x,y)\mapsto p_t(x,y)$ is locally H\"older continuous on $X\times X$ for each $t>0$, and that $t\mapsto p_t(x,y)$ is analytic in $(0,+\infty)$  for each $x,y\in X$.
	Moreover,  it was showed in \cite{JLZ16} that there exist constants $C_1, C_2>0$ depending only on $N$ and $K$ such that
	\begin{equation}\label{HeatKE}
		\begin{aligned}
			\frac{1}{C_1\meas\left(B_{\sqrt{t}}(x)\right)}\exp&\left\{-\frac{d^2(x,y)}{3t}-C_2t\right\}\leq p_t(x,y) \\
			&\leq\frac{C_1}{\meas\left(B_{\sqrt{t}}(x)\right)}\exp\left\{-\frac{d^2(x,y)}{5t}+C_2t\right\}
		\end{aligned}
	\end{equation}
	for all $t>0$ and all $x,y\in X$,  and 
	\begin{equation}\label{HeatKEGrad}
		\abs{\nabla p_t(\cdot,y)}(x)\leq\frac{C_1}{\sqrt{t}\cdot\meas\left(B_{\sqrt{t}}(x)\right)}\exp\left\{-\frac{d^2(x,y)}{5t}+C_2t\right\}
	\end{equation}
	for all $t>0$ and $\meas$-a.e. $x,y\in X$.
	Combining Davies \cite[Theorem 4]{D97} and (\ref{HeatKE}), the estimate 
	\begin{equation}\label{HeatKELap}
		\abs{\frac{\partial}{\partial t }p_t(x,y)}=	\abs{\Delta p_t(\cdot,y)}(x)\leq\frac{C_1}{t\cdot\meas\left(B_{\sqrt{t}}(x)\right)}\exp\left\{-\frac{d^2(x,y)}{5t}+C_2t\right\}
	\end{equation}
	holds for all $t>0$ and $\meas$-a.e. $x,y\in X$. An important property of $p_t(x,y)$ is the stochastic  completeness (see \cite{Stu94}):
	\begin{equation}\label{  equ2.6}
		\int_X p_t(x,y)d\meas(y)=1,\quad \forall t>0\ {\rm and}\ x\in X.
	\end{equation}
	
	The class of {\textit{test  functions}} on $RCD(K,N)$ spaces was introduced in \cite{GP20}:
	\begin{equation*}
		\mathrm{Test}^\infty(X):=\left\{f\in D(\Delta)\cap L^\infty(X):\abs{\nabla f}\in L^\infty(X),\ \Delta f\in L^\infty(X)\cap W^{1,2}(X)\right\},
	\end{equation*}
	which is an algebra and is dense in $W^{1,2}(X)$ (see \cite[Sect. 6.1.3]{GP20}).  Given any  $\phi\in \mathrm{Test}^\infty(X)$, it is well-known \cite{AMS16} that it has a (Lipschitz) continuous representative. Hence we always assume that every function in  $  \mathrm{Test}^\infty(X)$ is continuous. The following lemma is a standard fact about the heat kernel. 
	
	\begin{lemma}\label{lem2.1} {\rm (1)} For any $t>0$ and $x_0\in X$, the function $p_t(x_0,\cdot)\in {\rm Test}^\infty(X)$.\\
		{\rm (2)}  For any $\phi\in{\rm Test}^\infty(X)$ and $t>0$, it holds $P_t\phi\in {\rm Test}^\infty(X).$
	\end{lemma}
	\begin{proof}
		The assertion (1) follows from $p_{t/2}(x_0,\cdot)\in L^\infty(X)\cap L^1(X)$, 
		$$p_t(x_0,x)=\big[P_{t/2}\big(p_{t/2}(x_0,\cdot)\big)\big](x)$$
		and   the estimates (\ref{HeatKE})-(\ref{HeatKELap}).
		
		For (2), since $\phi\in D(\Delta)\cap L^\infty(X)$ and $\Delta\phi\in L^\infty(X)$,  by the contraction of $P_t$ on $L^\infty(X)$,  we have $P_t\phi\in L^\infty(X)$ and  $\Delta (P_t\phi)=P_t(\Delta\phi)   \in L^\infty(X)$.   From $\Delta\phi\in  L^2(X)$, we obtain $\Delta (P_t\phi)=P_t(\Delta\phi)\in W^{1,2}(X).$
		
		Finally,  since $|\nabla \phi|\in L^\infty(X)$, by using the Bakry-Ledoux inequality on $RCD(K,N)$ spaces (see \cite[Theorem 4.3]{EKS15}) and the contraction of $P_t$ on $L^\infty(X)$ again, we conclude that
		$$|\nabla P_t\phi|^2\leqslant e^{-2Kt} P_t(|\nabla \phi|^2)\in L^\infty(X).$$ 
		The proof of (2) is finished.
	\end{proof}

	The existence of good cut-off functions is an important  ingredient in geometric analysis.  For an $n$-dimensional Riemannian manifold $M^n$ with Ricci curvature   bounded from below,  one can estimate merely the upper bounds for the Laplacian of distance functions, by the Laplace comparison theorem. In this aspect,  Schoen-Yau in \cite[Theorem 4.2 in Chapter I]{SY94} constructed distance-like functions on $M^n$ having $L^\infty$-bounds of the Laplacian.  Cheeger-Colding \cite{CC96} found   good cut-off functions with $L^\infty$ estimate on the  Laplacian, which is a key technical tool in the Cheeger-Colding theory of
Ricci limit spaces (see \cite{CC96,CC97,CC00}). 
	 	The existence of a regular cut-off function  in $RCD$ setting was proved in \cite{AMS16, HKX13, MN19}. 
	  We need the following variant  of   \cite[Lemma 3.1]{MN19}.
	\begin{lemma}\label{cutoff}
		Let $(X,d,\meas)$ be an $RCD(K,N)$ space with $K\in\setR$ and $1\leq N<\infty$. 
		For each $r_0>0$, there exists a constant $C_{K,N,r_0}>0$ such that  for any ball $B_R(x)\subset X$ with radius  $R\geqslant r_0$, there exists a cut-off function $\eta: X\to \left[0,1\right]$ such that
		
		$\mathrm{(i)}$ $\eta\equiv1$ on $B_{R}(x)$ and ${\rm supp}(\eta)\subset  B_{2R}(x)$,
		
		$\mathrm{(ii)}$ $\eta\in \mathrm{Test}^\infty(X)$, and moreover  for $\meas$-a.e. $x\in X$,
		\begin{equation*} 
			|\Delta\eta|(x)+|\nabla\eta|(x)\leqslant \frac{C_{N,K,r_0}}{R}.
		\end{equation*}
	{\rm (}Remark that the constant $C_{N,K,r_0}$ does not depend on  $R$.{\rm )}
	\end{lemma}
	\begin{proof} This is  a slight modification  of \cite[Lemma 3.1]{MN19}.   For the convenience of readers, we give the proof  as follows. Fix a ball $B_{2R}(x)\subset X$ with radius $R\geqslant r_0$ and let $ \phi$ be a Lipschitz continuous function defined as $ \phi\equiv1$ on $B_{R}(x)$, $ \phi\equiv0$ on $X\setminus B_{2R}(x)$ and $ \phi(y)=\frac{2R-d(x,y)}{R}$ on $B_{2R}(x)\setminus B_{R}(x)$. As in \cite[Lemma 3.1]{MN19},  we consider the heat flow regularization $\phi_t:=P_t\phi$. According to \cite{EKS15},  we can choose
continuous representatives of $\phi_t, |\nabla \phi_t|$ and $|\Delta \phi_t|$; and moreover, by the Bakry-Ledoux's inequality (see  \cite[Theorem 4.3]{EKS15}), we have  for everywhere 
		\begin{equation}\label{equ2.7}
	|\nabla \phi_t|^2+\frac{4Kt^2}{N(e^{2Kt}-1)}|\Delta \phi_t|^2\leqslant e^{-2Kt}P_t(|\nabla \phi|^2)\leqslant \frac{e^{-2Kt}}{R^2},
		\end{equation}
	where we have used $|\nabla \phi|\leqslant 1/R$ and  $L^\infty$-contraction of $P_t$. 
 It follows that
$$|\phi_t(y)-\phi(y)|\leqslant \frac{1}{R}\int_0^t \sqrt{\frac{N(e^{2Ks}-1)e^{-2Ks}}{4Ks^2}}ds:=\frac{F_{K,N}(t) }{R}\leqslant  \frac{F_{K,N}(t) }{r_0},\quad \forall y\in X,$$
by $R\geqslant r_0$, where $F_{K,N}(\cdot):\mathbb R^+\to \mathbb R^+$ is continuous and $\lim_{t\to0^+}F_{K,N}(t)=0. $ 
It is clear that there exists $t_{K,N,r_0}>0$ such that $F_{K,N}(t_{K,N,r_0})\leqslant r_0/4$. Therefore  we have $\phi_{t_{K,N,r_0}}(y)\in [3/4,1] $ for every $y\in B_{R}(x)$ and $\phi_{t_{K,N,r_0}}(y)\in [0,1/4] $ for every $y\in X\setminus  B_{2R}(x)$. As in \cite[Lemma 3.1]{MN19}, the desired cut-off
function $\eta$ can be given by composition with a $C^2$-function $f:  \mathbb R \to [0, 1]$ such that $ f\equiv1$ on  
$[3/4, 1]$ and $f\equiv0$   on $[1/4, 0]$; indeed, $\eta:=f\circ\phi_{t_{K,N,r_0}}$  is now identically equal to one
on $B_{R-r_0}(x)$, vanishes identically on $X \setminus B_R(x)$  and, by \eqref{equ2.7} and the Chain rule (see for instance \cite[Proposition 5.2.3]{GP20}), it satisfies
the estimate $|\nabla \eta| + |\Delta\eta | \leqslant C(K,N,r_0)/R$ as desired.
 \end{proof}

	We recall the following measure-valued Laplacian  introduced in \cite{Gig15}.
	\begin{definition}\label{measured-laplace}
		Let $\Omega\subset X$ be an open subset and let  $f\in W_{\rm loc}^{1,2}(\Omega)$. It is called $f\in D(\boldsymbol{\Delta})$ on $\Omega$ if there exists a signed Radon measure $\mu$ on $\Omega$ such that 
		\begin{equation*}
			\int_\Omega \phi d\mu=-\int_\Omega\ip{\nabla \phi}{\nabla f} d\mathfrak{m},\quad \forall \phi\in  Lip_c(\Omega).		
		\end{equation*} 
  Such $\mu$ is unique and is denoted by $\boldsymbol{\Delta} f$ (which depends on $\Omega$).
	\end{definition}
Both the Chain rule and the Leibniz rule for this ${\bf \Delta}$ were proved in \cite[Chapter 4]{Gig15}.	
	When $\Omega=X$, it was showed \cite[Proposition 4.24]{Gig15} that for any $f\in W^{1,2}(X)$,
	$$f\in D(\Delta) \quad \Longleftrightarrow\quad f\in D({\bf\Delta})\ {\rm and}\ {\bf \Delta}f=g\cdot\meas\ {\rm for \ some }\ g\in L^2(X).$$ In this case ${\bf \Delta }f=\Delta f\cdot \meas$.

	The combination of \cite[Theorem 1.5]{ZZ16} and \cite[Theorem 3.1]{JKY14} states the following.
	\begin{lemma}\label{lem2.4}
		Let $\Omega\subset X$ be a domain and $f\in W^{1,2}_{\rm loc}(\Omega)$.  Suppose that 
		$\mathbf{\Delta} f=g\cdot \meas$ on $\Omega$ for some $g\in L^\infty_{\rm loc}(\Omega)$. Then $f\in  Lip_{\rm loc}(\Omega)$, and moreover for any ball $B_R(x)\subset \subset\Omega$, it holds
		$$\sup_{B_{R/2}(x)}|\nabla f| \leqslant C_{N,K,R}\big(\|f\|_{L^1(B_R(x))}+\|g\|_{L^\infty(B_R(x))}\big)$$
		for some constant $C_{N,K,R}.$ In particular, any   harmonic function $f$ (namely, ${\bf \Delta}f=0$ on $\Omega$) is locally Lipschitz continuous.
	\end{lemma}
	
At last, we recall a simple fact as follows.
\begin{lemma}\label{lem2.5}
Let $\Omega\subset X$ be  open. If $f\in D({\bf \Delta})\cap C(\Omega)$ then  ${\bf \Delta}f\big(\Omega\setminus {\rm supp}(f)\big)=0$. In particular, if $\Omega=X$ and if $f\in D(\Delta)\cap C(X)$ then $\Delta f(x)=0$ for $\meas$-a.e. $x\not\in {\rm supp}(f)$.  	
\end{lemma}
\begin{proof} Given any $\phi\in Lip_c(\Omega)$ with ${\rm supp}(\phi)\subset \Omega\setminus{\rm supp}(f)$,  by applying \cite[Corollary 2.25]{Che99}, we have that 
$|\nabla \phi|=0$ $\meas$-a.e. in $\Omega\setminus{\rm supp}(\phi)\supset {\rm supp}(f)$ and that $|\nabla f|=0$ $\meas$-a.e. in $\Omega\setminus{\rm supp}(f)$. It follows that $\ip{\nabla f}{\nabla \phi}=0$ $\meas$-a.e. in $\Omega$. Indeed, Cauchy-Schwarz inequality in \cite[Appendix]{Stu94} implies 
\begin{equation}\label{ equ2.8}
|\ip{\nabla f}{\nabla \phi}|\leqslant|\nabla f|\cdot|\nabla \phi| \quad  \meas{\rm-a.e.\ in}\ \Omega.
\end{equation}
Therefore, by the definition of ${\bf \Delta}f$, we have
${\bf \Delta}f(\phi)=0,$ which shows the first assertion.

If $\Omega=X$ and $f\in D(\Delta)$, by \cite[Proposition 4.24]{Gig15}, we have ${\bf \Delta}f=\Delta f\cdot\meas.$ Applying the first assertion, we have $\Delta f(x)=0$ for $\meas$-a.e. $x\in X\setminus{\rm supp}(f).$ The proof is finished.
\end{proof}

	\section{Distributions and Weyl's lemma} \label{sect-weyl-lem}
	Let $(X,d,\meas)$ be an $RCD(K,N)$ space for some $K\in \mathbb R$ and $N\in[1,+\infty)$. Recall that the class of {\textit{test  functions}}   
	\begin{equation*}
		\mathrm{Test}^\infty(X):=\left\{\phi\in D(\Delta)\cap L^\infty(X):\abs{\nabla \phi}\in L^\infty(X),\ \Delta \phi\in L^\infty(X)\cap W^{1,2}(X)\right\},
	\end{equation*}
  is dense in $W^{1,2}(X)$ and is an algebra (see \cite[Sect 6.1]{GP20}).
	We equip  $\mathrm{Test}^\infty(X)$  with the norm
	\begin{equation*}
		\|\phi\|_E:= \|\phi\|_{L^\infty(X)}+\||\nabla\phi|\|_{L^\infty(X)}+\|\Delta \phi\|_{L^\infty(X)},
	\end{equation*}
	and denote by $E(X)$ the linear normed space $\left(\mathrm{Test}^\infty(X), \|\cdot\|_E\right)$.

	Let $\Omega\subset X$ be an open domain, we set
	\begin{equation*}
	\begin{split}
	\mathrm{Test}_c^{\infty}(\Omega)&:=\left\{\phi \in  \mathrm{Test}^{\infty}(X) \big|\  {\rm    supp}(\phi)\subset \Omega \ {\rm and\ supp}(\phi) \ {\rm is\ compact}\right\}\\
		{\rm and}\quad \ \mathrm{Test}_{c,+}^{\infty}(\Omega)&:= \{\phi \in  \mathrm{Test}_c^{\infty}(\Omega)|\ \phi\geqslant 0\}.
		\end{split}
		\end{equation*}
		 	
		 From Lemma \ref{cutoff}, for each open domain $\Omega\subset X$ and every compact set $Q\subset \Omega$, there exists   $\eta\in \mathrm{Test}_c^{\infty}(\Omega)$ such that $\eta\equiv 1$ on $Q$ and $0\leqslant \eta\leqslant 1$.
  \begin{lemma}    \label{lem-add1}  
  The space $\mathrm{Test}_{c,+}^{\infty}(\Omega)$ is dense in $\{f\in Lip_c(\Omega)| f\geqslant 0\}$. Consequently, 
the space $\mathrm{Test}_c^{\infty}(\Omega)$ is dense in $W^{1,2}_0(\Omega)$. 
 \end{lemma}
\begin{proof}
This lemma is somewhat known. For the completeness, we present a proof here. 
 Let $f\in  Lip_c(\Omega)$ be nonnegative. Here $f$ is taken $0$ by the definition on $X\setminus \Omega$. According to \cite[Proposition 6.1.8]{GP20}), we take a sequence of nonnegative $f_j\in {\rm Test}^\infty(X)$ such that $f_j\to f$ in $W^{1,2}(X)$, as $j\to+\infty$.

Since ${\rm supp}(f)\subset \Omega$ is compact, we have $d\big({\rm supp}(f),X\setminus \Omega\big)>0.$ According to Lemma \ref{cutoff}, there exists a cut-off function $\eta\in {\rm Test}_c^\infty(\Omega)$, $0\leqslant \eta\leqslant 1$ and   $\eta\equiv1$, on some domain containing ${\rm supp}(f)$. Then we have $f\eta\equiv f$ on $X$ and $f_j\eta\in {\rm Test}^\infty_c(\Omega)$ for all $j$. By 
 $|f_j\eta-f|=|(f_j-f)\eta|\leqslant |f_j-f|$ and 
 $$|\nabla (f_j\eta-f)|^2 \leqslant 2|\nabla (f_j-f)|^2\eta^2+2 (f_j-f)^2|\nabla \eta|^2,$$
We conclude that $f_j\eta\to f$ in $W^{1,2}(X)$ as $j\to+\infty$, since $\eta\in Lip(X)$. Notice that $f_j\eta\geqslant 0$. The proof of the first assertion is finished.

The second assertion follows from the first one by noticing the facts   that the space $Lip_c(\Omega)$ is dense in $W^{1,2}_0(\Omega)$ and that for any $f\in Lip_c(\Omega)$ it holds $f=f^+-f^-$ and $f^{\pm}\in Lip_c(\Omega)$.
\end{proof}

	\begin{definition}Let $\Omega\subset X$. 
		A {\it distribution } $T$ in $\Omega$ is a linear functional on $\mathrm{Test}_c^{\infty}(\Omega)$ such that for every compact set $Q\subset \Omega $ there exists a constant
		$C_{Q}$ such that 
		$$|T(\phi)|\leqslant C_{Q}  \|\phi\|_E,\quad \forall  \phi\in \mathrm{Test}_c^{\infty}(\Omega)\  {\rm  with}\ {\rm supp}(\phi)\subset Q.$$
	\end{definition}
	
	It is clear that  every
	element $f\in E'(X)$, the dual space of $E(X)$,     defines a distribution on $X$.

	\begin{definition} Let $f\in L^1_{\rm loc}(\Omega)$, the \emph{distributional Laplacian} of $f$, denoted by $\Delta_{\mathscr D,\Omega}f$,  is defined as a linear functional on $\mathrm{Test}_c^{\infty}(\Omega)$ by
		\begin{equation*}
			\Delta_{\mathscr D,\Omega}f(\phi):=\int_\Omega f\Delta\phi d \meas,\quad \forall \phi\in  \mathrm{Test}_c^{\infty}(\Omega).
		\end{equation*}
		If $\Omega=X$, we denote $\Delta_{\mathscr D}f:=\Delta_{\mathscr D,X}f$ for any $f\in L^1_{\rm loc}(X)$. 
\end{definition}
		
		For each $f\in L^1_{\rm loc}(\Omega)$, it is easy to check that $\Delta_{\mathscr D,\Omega}f$ is a distribution in $\Omega$. 
Given two domains   $\Omega\subset \Omega'$ and a function $f\in L^1_{\rm loc}(\Omega')$, it is not hard to check that $\Delta_{\mathscr D,\Omega}f=\Delta_{\mathscr D,\Omega'}f$ in $\Omega$. In fact, for any   $\phi\in {\rm Test}^\infty_c(\Omega)$,  Lemma \ref{lem2.5} implies that $\Delta \phi(x)=0$ for $\meas$-a.e. $x \in \Omega'\setminus\Omega$. It follows
 that $\int_{\Omega'}f\Delta\phi d\meas=\int_{\Omega}f\Delta\phi d\meas$.

	First of all,  we show that the distribution Laplacian is compatible with the measure-valued Laplacian in \cite{Gig15}, and also, when $\Omega=X$,   with the generator of $L^p(X)$-semigroup $P_t$ for each $p\in[1,\infty)$.
	\begin{lemma}\label{lem3.3} Let $\Omega\subset X$ be an open domain. 
		If $f\in W^{1,2}_{\rm loc}(\Omega)$, then $\Delta_{\mathscr D,\Omega}f= {\bf \Delta}f$ on $  \mathrm{Test}^\infty_c(\Omega).$ In particular,  if $u$ is  subharmonic/superharmonic on $\Omega$ then it is very weak subharmonic/superharmonic on $\Omega$.
		\end{lemma}
	\begin{proof}
		Take any $\phi\in\mathrm{Test}^\infty_c(\Omega).$ Since $d\big({\rm supp}(\phi), X\setminus\Omega\big)>0$, there exists a Lipschitz function $\chi:X\to [0,1]$ such that  $\chi\equiv 1$ on ${\rm supp}(\phi)$ and ${\rm supp}(\chi)\subset \Omega$. Noticing that ${\rm supp}(\phi)$ is compact, we can choose $\chi$ such that ${\rm supp}(\chi)$ is compact too. 	
				From the definition of $W^{1,2}_{\rm loc}(\Omega)$ and $f\in W^{1,2}_{\rm loc}(\Omega)$,  we have $h:=f\cdot \chi\in W^{1,2}(X)$; moreover, since $h(x)=f(x)$ for $\meas$-a.e.  $x\in{\rm supp}(\phi)$, by using \cite[Corollary 2.25]{Che99}, we get $|\nabla f|(x)=|\nabla h|(x)$ for $\meas$-a.e.   $x\in{\rm supp}(\phi)$ and $|\nabla \phi|(x)=0$ for $\meas$-a.e.   $x\not\in  {\rm supp}(\phi)$. It follows that (see  (\ref{ equ2.8}))			
 $$\ip{\nabla h}{\nabla \phi}=\ip{\nabla f}{\nabla \phi}\quad \meas{\rm-a.e.\  in}\ \ X.$$
By the definition of ${\bf \Delta}f$ (see Definition \ref{measured-laplace}), we obtain
		\begin{equation}\label{equ3.1}
				\begin{split}
			{\bf \Delta}f(\phi)&=-\int_\Omega\ip{\nabla f}{\nabla \phi}d \meas=-\int_X\ip{\nabla f}{\nabla \phi}d \meas\\
			&=-\int_X\ip{\nabla h}{\nabla \phi}d \meas.	
			\end{split}
		\end{equation}
Since $h\in W^{1,2}(X)$ and $\phi\in {\rm Test}^\infty(X)\subset D(\Delta)$, we have
$$\int_Xh\Delta\phi d\meas=-\int_X\ip{\nabla h}{\nabla \phi}d\meas.$$
Together this with (\ref{equ3.1}) and the facts that $f=h$ on ${\rm supp}(\phi)$ and that $\Delta\phi(x)=0$ for $\meas$-a.e. $x\not\in {\rm supp}(\phi)$ (see  Lemma \ref{lem2.5}), we conclude that  
		\begin{equation*}
			\begin{split}
				{\bf \Delta}f(\phi)&=\int_Xh\Delta \phi d \meas=\int_{{\rm supp}(\phi)}h\Delta \phi d \meas \\
				&=\int_{{\rm supp}(\phi)}f\Delta\phi d \meas=\int_\Omega f\Delta \phi d \meas=\Delta_{\mathscr D,\Omega}f(\phi).
			\end{split}
		\end{equation*}
		The proof is finished.
	\end{proof}
	
	\begin{lemma}\label{lem3.4}
		For each $p\in[1,\infty)$, if $f\in D(\Delta^{(p)})$, then  $\Delta_{\mathscr D}f= \Delta^{(p)}f$ in the sense of distributions on $X$. Here
		$\Delta^{(p)}$ is the generator of $P_t$ as semi-group acting on $L^p(X)$.  That is,
		\begin{equation*}
		\begin{split}D(\Delta^{(p)})&:=\left\{f\in L^p(X)\Big| \lim_{t\to0^+}\frac{P_tf-f}{t}\ {\rm exists\ in}\ L^p\right\}\ \quad {\rm and}\\
		 \Delta^{(p)}f&:=\lim_{t\to0^+}\frac{P_tf-f}{t}.
		 \end{split}
		 \end{equation*}
	\end{lemma}
	
	\begin{proof}
		Take any $\phi\in {\rm Test}_c^\infty(X)$. Since $\Delta\phi\in L^2(X)\cap L^\infty(X)$, we have
		$$\left\|\frac{P_t\phi-\phi}{t}\right\|_{L^\infty(X)}\leqslant\frac 1 t\int_0^t\|P_s(\Delta \phi)\|_{L^\infty(X)}ds\leqslant \|\Delta\phi\|_{L^\infty(X)},$$
		where we have used the contraction of $P_s$ on $L^\infty(X)$ for any $s>0$. Notice that $\frac{P_t\phi-\phi}{t}\to \Delta\phi$ in $L^2$ as $t\to0^+$ implies that there exist a sequence $t_j\to0^+$ such that $\frac{P_{t_j}\phi(x)-\phi(x)}{t_j}\to \Delta\phi(x)$ for $\meas$-a.e. $x\in X$. 		
		Since $f\in L^p(X)$, the dominated converging theorem implies that
		\begin{equation} \label{equ3.2}
			\int_Xf\frac{P_{t_j}\phi-\phi}{t_j}d\meas\ \to \ \int_Xf\Delta \phi d\meas,\quad {\rm as}\ \ t_j\to0^+.
		\end{equation}
		On the other hand, recalling that $f\in D(\Delta^{(p)})$ states $\frac{P_tf-f}{t}\to \Delta^{(p)}f$ in $L^p(X)$ as $t\to0^+$, and by using $p_t(x,y)$ is symmetric, we get
		$$\int_X f\frac{P_t\phi-\phi}{t} d\meas=\int_X\frac{P_tf-f}{t} \phi d\meas\ \to\ \int_X\Delta^{(p)}f\cdot \phi d\meas\quad {\rm as}\ \ t\to0^+. $$
		Therefore, by combining with (\ref{equ3.2}), we conclude that
		$$\int_Xf\Delta \phi d\meas=\int_X\Delta^{(p)}f\cdot \phi d\meas,\quad \forall \phi\in {\rm Test}^\infty_c(X).$$
		This is $\Delta_{\mathscr D}f=\Delta^{(p)}f$ in the sense of distributions.
	\end{proof}

	We can now show the following special case of  Weyl's lemma for $L^1$ functions with compact support, which is the core of this paper.
	
	\begin{lemma}\label{lem3.5}
		Let $v\in L^1(X)$ such that $v(x)=0$ a.e. $x\in X\setminus B_R(x_0)$ for some ball $B_R(x_0)\subset X.$ Suppose that  $\Delta_{\mathscr D}v(\phi)=0$ for any nonnegative $\phi\in \mathrm{Test}^\infty_c(X)$   with ${\rm supp}(\phi)\subset B_{R/2}(x_0)$.
		Then $v\in  Lip\big(B_{R/8}(x_0)\big)$.
	\end{lemma}
	\begin{proof}
		Let   $p_t(x,y)$ be the heat kernel on $X$. For any $t>0$, we set
		\begin{equation*} 
			v_t(x):=P_tv(x)= \int_Xp_t(x,y)v(y)d \meas,\quad \forall t>0.
			\nonumber
		\end{equation*}
		{\it Claim 1:   For each $t\in(0,R^2)$, we have $v_t\in   L^\infty(X)\cap W^{1,2}(X)$.}
		
		Fix any $x\in X$. Since $v(y)=0$ for $\meas$-a.e.  $y\in X\setminus B_R(x_0)$, by the upper bound of heat kernel in \eqref{HeatKE}, we have for any $x\in X$ that
		\begin{equation}\label{ equ3.3}
			|v_t(x)|\leqslant  \int_{B_R(x_0)}\frac{C_1|v(y)|}{\meas(B_{\sqrt t}(y))}d\meas(y)\leqslant \frac{C_1}{\inf_{y\in B_R(x_0)}\meas(B_{\sqrt t}(y))}\cdot\|v\|_{L^1(X)}.
		\end{equation}
		Here and in the following, we denote $C_1, C_2, \cdots $ for varying constants depending only on $K,N$, and $R$.
		The local measure doubling property \eqref{equ2.1} implies for $t<R^2$ 
		$$\inf_{y\in B_R(x_0)}\meas(B_{\sqrt t}(y))\geqslant \inf_{y\in B_R(x_0)}C_2\meas(B_{2R}(y)  )\cdot \Big(\frac{\sqrt t}{2R}\Big)^N \geqslant C_3\meas(B_{R}(x_0)  )\cdot t^{N/2}.$$
		By combining this with \eqref{ equ3.3}, we get $v_t\in L^\infty(X)$ for any $t\in(0,R^2).$

		Since $P_t$ is contraction on $L^1$, we have $P_{t/2}v\in L^1(X)$. By combining with $P_{t/2}v\in L^\infty(X)$, we have $P_{t/2}v\in L^2(X)$. Hence, we obtain
		$v_t=P_{t/2}(P_{t/2}v)\in W^{1,2}(X)$. Now the proof of the Claim is completed.\\

		\noindent {\it Claim 2: There exists a constant $c>0$ such that for any $t\in (0,R^2)$ it holds  
			\begin{equation}\label{ equ3.4}
				\int_{B_{R/4}(x_0)}\ip{\nabla v_t}{\nabla \phi} d\meas\leqslant c\int_{B_{R/4}(x_0)}|\phi|d\meas,\quad \forall\phi\in Lip_c\big(B_{R/4}(x_0)\big).
		\end{equation}}

		Take any  nonnegative $\phi\in \mathrm{Test}^\infty_c\big(B_{R/4}(x_0)\big)$. Since $\phi\in D(\Delta)$  and  $v_t\in W^{1,2}(X)$ (from {\it Claim 1}), by  using \eqref{equ2.2} and $\Delta (P_t\phi)=P_t( \Delta \phi)$, we have
		\begin{equation}\label{ equ3.5}
			\begin{split}
				\int_{B_{R/4}(x_0)}\ip{\nabla v_t}{\nabla \phi}d\meas&=\int_X\ip{\nabla v_t}{\nabla \phi}d\meas=-\int_Xv_t\Delta\phi d\meas\\
				&=-\int_XvP_t(\Delta \phi)d\meas=-\int_Xv\Delta (P_t\phi)d\meas.
			\end{split}
		\end{equation}
	According to Lemma \ref{cutoff}, there exist two 	 cut-off functions  $\eta_{i}:X\to [0,1] $  in ${\rm Test}^\infty(X)$, $i=1,2$, such that:
		
		(i) ${\rm supp}(\eta_1)\subset B_{R/2}(x_0)$ and $\eta_1\equiv1$ on $B_{R/3}(x_0)$;
		
		(ii) ${\rm supp}(\eta_2)\subset B_{2R}(x_0)$ and $\eta_2\equiv1$ on $B_{3R/2}(x_0)$; and 
		\begin{equation}\label{ equ3.6}
			|\nabla \eta_i|+|\Delta \eta_i|\leqslant C_{4}, \quad i=1,2.
		\end{equation}
		From Lemma \ref{lem2.1}, we know that   $\eta_i\cdot P_t\phi\in {\rm Test}_c^\infty(X)$ for $i=1,2$. Since $ \eta_2\cdot P_t\phi=P_t\phi$  in $B_{3R/2}(x_0)$, Lemma \ref{lem2.5} implies that   $\Delta(\eta_2\cdot P_t\phi)=\Delta(P_t\phi)$ a.e. in $B_{3R/2}(x_0)$. Combining with the assumption that  $v=0$ a.e. outside $B_R(x_0)$, we have
		\begin{equation}\label{equ3.7}
			\int_Xv\Delta(P_t\phi)d\meas=\int_Xv\Delta(\eta_2\cdot P_t\phi)d\meas=\Delta_{\mathscr D}v(\eta_2\cdot P_t\phi).
		\end{equation}
		Since  $\eta_1\cdot P_t\phi\in {\rm Test}_c^\infty(X)$ and ${\rm supp}(\eta_1\cdot P_t\phi)\subset B_{R/2}(x_0)$, by the assumption $\Delta_{\mathscr D}v =0$ on $B_{R/2}(x_0)$,  we get
		$$\Delta_{\mathscr D}v(\eta_1\cdot P_t\phi)=0.$$
		Therefore, by substituting this and (\ref{equ3.7}) into (\ref{ equ3.5}), we obtain that 
		\begin{equation}\label{equ3.8}
			\begin{split}
				\int_{B_{R/4}(x_0)}\ip{\nabla v_t}{\nabla \phi}d\meas&=\Delta_{\mathscr D}v\big((\eta_2-\eta_1)\cdot P_t\phi\big)\\
				&=\int_X v\Delta\big(\eta \cdot P_t\phi\big)d\meas,
			\end{split}
		\end{equation}
		where $\eta:=\eta_2-\eta_1\in {\rm Test}^\infty_c(X)$.
		
		Next we want to estimate $\|\Delta(\eta\cdot P_t\phi)\|_{L^\infty(X)}$.
		Since ${\rm supp}(\eta)\subset \overline{B_{2R}(x_0)\setminus B_{R/3}(x_0)}$, we know, by Lemma \ref{lem2.5}, that 
		\begin{equation}\label{equ3.9}
			\Delta(\eta\cdot P_t\phi)=0\ \ \meas{\rm-a.e. \ in}\  B_{R/3}(x_0)\cup B^c_{2R}(x_0).
		\end{equation}
		Fix any $x\in \overline{B_{2R}(x_0)\setminus B_{R/3}(x_0)}$. Since $d(x,y)\geqslant \frac{R}{12}$ for any $y\in {\rm supp}(\phi)$, by using the estimates (\ref{HeatKE}) and ${\rm supp}(\phi)\subset B_{R/4}(x_0)$, we  obtain
		\begin{equation}\label{equ3.10}
			\begin{split}
				|P_t\phi(x)|&\leqslant \int_{B_{R/4}(x_0)}\frac{C_5\cdot \exp\left(-\frac{R^2}{5t\cdot 12^2}\right)}{\meas(B_{\sqrt t}(y))}|\phi(y)|d\meas(y)\\
				&\leqslant C_6 \Big(\frac{R}{\sqrt t}\Big)^N\cdot \exp\left(-\frac{R^2}{5t\cdot 12^2}\right)\int_{B_{R/4}(x_0)}\frac{|\phi(y)|}{\meas(B_{R}(y))} d\meas(y)\\
				&\leqslant \frac{C_7}{\meas(B_{R/2}(x_0))}\|\phi\|_{L^1(X)},
			\end{split}
		\end{equation}
		where in the second inequality we have used (\ref{equ2.1}), and in the last inequality  we have used  $B_R(y)\supset B_{R/2}(x_0)$ for any $y\in B_{R/4}(x_0)$ and that 
		\begin{equation*} 
			\frac{1}{t^{N/2}}\cdot \exp\left(-\frac{R^2}{5t\cdot 12^2}\right)\leqslant C_{N,R},\quad \forall t\in(0,R^2)
		\end{equation*}
		for some constant $C_{N,R}>0.$ Similarly, by (\ref{HeatKEGrad}) and (\ref{HeatKELap}), we get that 
		\begin{equation}\label{equ3.11}
			\begin{split}
				|\nabla P_t\phi(x)|& \leqslant \int_{B_{R/4}(x_0)}\big|\nabla p_t(x,y)\big|\cdot|\phi(y)|d\meas(y)\\
				& \leqslant \int_{B_{R/4}(x_0)}\frac{C_8\cdot \exp\left(-\frac{R^2}{5t\cdot 12^2}\right)}{\sqrt t\cdot\meas(B_{\sqrt t}(y))}|\phi(y)|d\meas(y)\\
				&\leqslant \frac{C_9}{\sqrt t} \Big(\frac{R}{\sqrt t}\Big)^N\cdot \exp\left(-\frac{R^2}{5t\cdot 12^2}\right)\int_{B_{R/4}(x_0)}\frac{|\phi(y)|}{\meas(B_{R}(y))} d\meas(y)\\
				&\leqslant \frac{C_{10}}{\meas(B_{R/2}(x_0))}\|\phi\|_{L^1(X)}
			\end{split}
		\end{equation}
		and
		\begin{equation}\label{equ3.12}
			\begin{split}
				|\Delta P_t\phi(x)|&=\big|\frac{\partial}{\partial t}P_t\phi(y)\big|\leqslant \int_{B_{R/4}(x_0)}\big|\frac{\partial}{\partial t}p_t(x,y)\big|\cdot|\phi(y)|d\meas(y)\\
				& \leqslant \int_{B_{R/4}(x_0)}\frac{C_{11}\cdot \exp\left(-\frac{R^2}{5t\cdot 12^2}\right)}{t\cdot\meas(B_{\sqrt t}(y))}|\phi(y)|d\meas(y)\\
				&\leqslant \frac{C_{12}}{\meas(B_{R/2}(x_0))}\|\phi\|_{L^1(X)},
			\end{split}
		\end{equation}
		for almost every  $x\in \overline{B_{2R}(x_0)\setminus B_{R/3}(x_0)}$.
	Noticing that $\eta, P_t\phi\in {\rm Test}^\infty(X)$,  by using Leibniz rule   (see \cite[Proposition 5.2.3]{GP20}) to $\Delta(\eta\cdot P_t\phi)$, and then substituting  (\ref{ equ3.6}),  (\ref{equ3.10}), (\ref{equ3.11}) and (\ref{equ3.12}), we get that 
		\begin{equation*} 
			|\Delta(\eta\cdot P_t\phi)|(x)\leqslant \frac{C_{13}}{\meas(B_{R/2}(x_0))} \|\phi\|_{L^1(X)} 
		\end{equation*}
		for almost all $x\in \overline{B_{2R}(x_0)\setminus B_{R/3}(x_0)}$. 
		By combining with (\ref{equ3.9}), we conclude 
		\begin{equation*}
			\|\Delta(\eta\cdot P_t\phi)\|_{L^\infty(X)}\leqslant \frac{C_{13}}{\meas(B_{R/2}(x_0))} \|\phi\|_{L^1(X)} ,\quad \forall t\in(0,R^2).
		\end{equation*}
		Substituting this into  (\ref{equ3.8}), we get that  for $\phi\in \mathrm{Test}^\infty_{c,+}\big(B_{R/4}(x_0)\big)$,
				\begin{equation} \label{equ3.13}
			\int_{B_{R/4}(x_0)}\ip{\nabla v_t}{\nabla \phi}d\meas\leqslant \frac{C_{13}\cdot \|v\|_{L^1(X)}}{\meas(B_{R/2}(x_0))} \|\phi\|_{L^1(X)}:=c\cdot \|\phi\|_{L^1(X)}.
		\end{equation}  
		By using Lemma \ref{lem-add1}, we know that   (\ref{equ3.13}) holds too for any
 nonnegative $\phi\in Lip_c(\Omega)$. For any $\phi\in Lip_c(\Omega)$, by applying (\ref{equ3.13}) to $\phi^{\pm}$, the desired (\ref{ equ3.4}) follows.    Now the proof of {\it Claim 2} is finished.

		Now, from the Definition \ref{measured-laplace}, {\it Claim 2} states that $v_t\in D(\mathbf{\Delta})$ on $B_{R/4}(x_0)$  and  
		${\bf \Delta}v_t=g_t\cdot\meas$  for some Borel function $g_t$ with  $\|g\|_{L^\infty(B_{R/4}(x_0))}\leqslant c$. According to Lemma \ref{lem2.4}, we get
		$$\sup_{B_{R/8}(x_0)} |\nabla v_t |\leqslant C_{14}\Big(\|v_t\|_{L^1\big(B_{R/4}(x_0)\big)}+c\Big).$$
		Since   $P_t$ is a contraction in $L^1(X)$, we have
		$$\sup_{B_{R/8}(x_0)} |\nabla v_t |\leqslant C_{14}\Big(\|v_t\|_{L^1(X)}+c\Big)\leqslant C_{14}(\|v\|_{L^1(X)}+c):=c_0.$$
		Notice  the constant $c_0$ is  independent of $t$.  Letting  $t\to 0^+$,   since $v_t\to v$ in $L^1(X)$ (by $v\in L^1(X)$), we get that 
		$$\sup_{B_{R/8}(x_0)} |\nabla v |\leqslant c_0.$$
		The proof of Lemma \ref{lem3.5} is finished.
	\end{proof}

	Now  we  are in the position to show Weyl's lemma in $RCD$-setting. 
	\begin{proof}[Proof of Theorem \ref{weyl-lemma}] 
		Let $B_R(x_0)\subset \subset\Omega$ be a ball. We take a Lipschitz function $\chi:X\to[0,1]$ such that $\chi\equiv1$ on $B_{R/2}(x_0)$ and $\chi \equiv0$ on $X\setminus B_{R}(x_0)$. Let 
		$v:=u\cdot \chi$ on $\Omega$ and $v=0$ on $X\setminus \Omega$. Now for any nonnegative $\phi\in {\rm Test}^\infty_{c}\big(B_{R/2}(x_0)\big)$, we have
		$$\Delta_{\mathscr D}v(\phi) =\int_Xv\Delta\phi d\meas=\int_{B_{R/2}(x_0)} u\Delta\phi d\meas=0, $$
		where we have used $\Delta \phi=0$ a.e. in $X\setminus B_{R/2}(x_0)$, since ${\rm supp}(\phi)\subset B_{R/2}(x_0)$ and  Lemma \ref{lem2.5}. By applying Lemma \ref{lem3.5}, we conclude that $v\in  Lip(B_{R/8}(x_0)).$ It follows that $u\in  Lip(B_{R/8}(x_0)),$ since $u=v$ on $B_{R/2}(x_0)$.  Now the proof of Theorem \ref{weyl-lemma} is finished.
	\end{proof}

	\section{Regularity of very weak solutions for Poisson equations}\label{sect4}
	Let $(X,d,\meas)$ be an $RCD(K,N)$ space for some $K\in \mathbb R$ and $N\in[1,+\infty)$.
	
	\begin{definition} \label{poisson}
		Let   $\Omega\subset X$ be an open domain. A function  $u\in L_{\rm loc}^1(\Omega)$ is called a {\it very weakly subsolution (resp. supersolution)} of Poisson equation $\Delta_{\mathscr D,\Omega}u =f$ on $\Omega$ for some $f\in L^1_{\rm loc}(\Omega)$ if 
	\begin{equation}\label{equ4.1}
	\int_{\Omega}u\Delta \phi d\meas\geqslant \ ({\rm resp}. \leqslant)\ \int_\Omega f \phi d\meas, \qquad \forall  \phi\in {\rm Test}^\infty_{c}(\Omega)\ {\rm and}\ \phi\geqslant 0.
	\end{equation}
A function  $u\in L_{\rm loc}^1(\Omega)$ is called a {\it very weakly solution} of $\Delta_{\mathscr D,\Omega}u = f$  on $\Omega$ if it is both very weakly subsolution and supersolution. 
\end{definition}

	Next we will show Corollary \ref{cor1.5}.
	\begin{proof}[Proof of Corollary \ref{cor1.5}]
		Take any  ball $B:=B_R(x)\subset\subset \Omega$. Since $f\in L^2_{\rm loc}(\Omega)$, we have $f\in L^2(B)$. Let $w$ be the solution of the  
		relaxed Dirichlet problem  (see \cite[Theorem 7.12]{Che99} for the solvability):  
		\begin{equation*}
			\begin{cases}
				{\bf \Delta}w=f\cdot\meas\quad  {\rm  on} \ \  B,\\
				w\in W^{1,2}_0(B).
			\end{cases}
		\end{equation*}
	By Lemma \ref{lem3.3}, we get that $\Delta_{\mathscr D,B}w={\bf \Delta}w$, i.e.,
	\begin{equation*}
	\int_Bw\Delta \phi d\meas={\bf \Delta }w(\phi)=\int_Bf\phi d\meas,\quad \ \forall \phi\in {\rm Test}^\infty_{c}(B) \ \ {\rm and}\ \ \phi\geqslant 0.
	\end{equation*}
	By combining with (\ref{equ1.5})  and Lemma \ref{lem2.5}, we have
	\begin{equation*}
	\int_B(u-w)\Delta \phi d\meas=0,\quad \ \forall \phi\in {\rm Test}^\infty_c(B) \ \ {\rm and}\ \ \phi\geqslant 0.
	\end{equation*}	
Therefore the function $u-w\in L^1(B)$ is a very weak harmonic function on $B$. From Theorem \ref{weyl-lemma}, we know that $u-w\in Lip_{\rm loc}(B)$.  It follows $u=(u-w)+w\in W^{1,2}_{\rm loc}(B).$  The proof is finished.
	\end{proof}

	\section{The   Liouville property for $L^1$ (sub)harmonic functions}\label{sect5}
	In this section,  we will prove the Liouville-type theorem for $L^1$ subharmonic/harmonic functions on $RCD(K,N)$ space, which is an extension of  Peter Li's result in  \cite{PeterLi84}. Let $(X,d,\meas)$ be an $RCD(K,N)$ space with $K\in\mathbb R$ and $N\in[1,+\infty)$.	
	
	\begin{lemma}\label{lem5.1}
	Let $u$ and $f$ be in $L^1(X)$. Suppose that $u$ is a very weakly subsolution of Poisson equation $\Delta_{\mathscr D}u=f$ on $X$. Then
		\begin{equation}\label{equ5.1}
			\int_{X} u\Delta h d\meas\geqslant  \int_Xfh d\meas,\quad\ \forall h\in \mathrm{Test}^{\infty}(X)\ {\rm and}\ h\geqslant 0.
					\end{equation}
			\end{lemma}
	\begin{proof}
	If $X$ is compact, then we have $\mathrm{Test}^{\infty}_c(X)=\mathrm{Test}^{\infty}(X)$. This assertion follows from the Definition \ref{poisson} with $\Omega=X$.  
	
	Suppose that $X$ is noncompact. Take any  $h\in \mathrm{Test}^{\infty}(X)$ and $h\geqslant 0$.
		For each $R> 1$, by using Lemma \ref{cutoff} (with $r_0=1$ therein),  we have a cut-off function $\eta: X\to [0,1]$   in ${\rm Test}^\infty(X)$  such that $\eta\equiv1$ in $B_{R}(x_0)$,  $\eta\equiv0$ outside $B_{2R}(x_0)$ and 	
\begin{equation}\label{equ5.3}
|\nabla \eta|+|\Delta \eta|\leqslant \frac{C_{K,N}}{R}\leqslant C_{K,N}\quad \meas{-a.e.\ in }\ X,
\end{equation}
where  the constant $C_{K,N}$ is independent of $R$. 
		We set  $\phi:=h\eta$, and then $\phi\in{\rm Test}^{\infty}_c(X)$. Since $\phi=h$ on $B_{R}(x_0)$, from Lemma \ref{lem2.5}, we have
	\begin{equation}\label{equ5.4}
	\Delta \phi(x)=\Delta h(x) \quad {\rm  for}\  \meas{\rm-a.e. }\  x\in B_{R}(x_0).
	\end{equation}
	Since $h,\eta\in {\rm Test}^\infty(X)$, by using  Leibniz rule   (see \cite[Proposition 5.2.3]{GP20}) to $\phi=h\eta$, we get that for $\meas$-a.e. $x\in X\setminus B_{R}(x_0)$,
		\begin{equation*}  
		\begin{split}
				|\Delta \phi|(x)&\leqslant h(x)|\Delta\eta|(x) +2|\ip{\nabla h}{\nabla \eta}|(x)+ \eta(x)|\Delta h|(x) \\
				& \leqslant 2C_{K,N} \big(\|h\|_{L^\infty(X)}+\||\nabla h|\|_{L^\infty(X)}\big)+\|\Delta h\|_{L^\infty(X)}:=c_{K,N,h},
		\end{split}
		\end{equation*}
where we have used 	 (\ref{equ5.3}). 
By combining  this with (\ref{equ5.4}), we get   
		\begin{equation*}
		\begin{split}
			\int_{X}u\Delta hd\meas - \int_{X}u\Delta \phi d\meas&= \int_{X\setminus B_R(x_0)}u(\Delta h-\Delta \phi)d\meas\\
			 &\geqslant -c'_{K,N,h} \cdot   \|u\|_{L^1(X\setminus B_R(x_0)},
			 		\end{split}
		\end{equation*}
		where $c'_{K,N,h}:=c_{K,N,h}+\|\Delta h\|_{L^\infty(X)}$. Since $\phi=h$ on $B_{R}(x_0)$, we have
$$\int_Xf\phi  d\meas-\int_Xfh d\meas\geqslant -\|h\|_{L^\infty(X)}\cdot\|f\|_{L^1(X\setminus B_R(x_0)}.$$
From the combination of these two inequalities and  (\ref{equ4.1}) with $\Omega=X$ (noticing that $\phi\in {\rm Test}^\infty_c(X)$), we get
$$\int_Xu\Delta hd\meas\geqslant \int_Xfh d\meas- c'_{K,N,h} \cdot   \|u\|_{L^1(X\setminus B_R(x_0))}-\|h\|_{L^\infty(X)}\cdot\|f\|_{L^1(X\setminus B_R(x_0))}.
$$	
Letting $R\to +\infty$, it follows that  (\ref{equ5.1}), since $u\in L^1(X)$ and $h\in L^1(X)$.
	\end{proof}
To show the $L^1$-Liouville property, we need  the following proposition.
\begin{proposition}\label{prop5.2}
Let   $u\in L^1(X)$ be a  very weakly subharmonic (or superharmonic) function on $X$. Then it is a very weak harmonic function on $X$. Furthermore,  $u$ is in $ Lip_{\rm loc}(X)$ and is a harmonic function on $X$, namely,
  $u\in D({\bf \Delta})$ and ${\bf \Delta}u=0$ on $X$.				 
 \end{proposition}
\begin{proof}
	 We first consider that  $u$  is a $L^1$  very weakly subharmonic function on $X$. 
		Since by (\ref{HeatKELap}), $\frac{\partial}{\partial t}p_t(x,\cdot)(y)\in{L^\infty(X)}$ for each $t\in (0,\infty)$, we have for $\meas$-a.e. $x\in X$ that 
		\begin{equation}
			\begin{aligned}
				\frac{\partial}{\partial t}P_t u(x)&=\int_{X} \frac{\partial}{\partial t}p_t(x,y)u(y)d\meas(y)\\
				&=\int_{X} \Delta_y p_{t}(x,y)u(y)d\meas(y). \\
			\end{aligned}
			\nonumber
		\end{equation}
	From Lemma \ref{lem2.1}, we have    $p_t(x,\cdot)\in\mathrm{Test}^{\infty}(X)$. Since   $u$ is a $L^1$   very weakly subharmonic function on $X$, by using Lemma \ref{lem5.1},  we conclude that $P_t u(x)$ is a monotone nondecreasing function in $t$, for almost all $x\in X$.
	  On the other hand,  By the stochastic completeness (\ref{  equ2.6}) and Fubini's theorem, we get
		$$\int_XP_tu(x) d\meas(x) =\int_X\int_Xp_t(x,y)u(y)d\meas(y)d\meas(x)=\int_X u(x)\meas (x).$$ 
Therefore,   we have 
		$$P_tu(x) =u(x) \quad \meas{\rm-a.e.}\ x\in X.$$
		This implies $u\in D(\Delta^{(1)})$ and $\Delta^{(1)}u=0$, since $u\in L^1(X)$. From Lemma \ref{lem3.4}, we know that $\Delta_{\mathscr D}u=0$ on $X$ in the sense of distributions. That is, $u$ is a very weak harmonic function on $X$.
		
	On the other hand, if $u$  be a $L^1$  very weakly superharmonic function on $X$, then $-u$ is a very weakly subharmonic, and hence $-u$ is a very weak harmonic function on $X$.

	Now we want to show the second statement. From the first one,  $u$ is very weak harmonic function on $X$.
	By applying Theorem \ref{weyl-lemma}, we get that $u\in  Lip_{\rm loc}(X)$.  Hence, by Lemma \ref{lem3.3}, we get that $u\in D({\bf\Delta})$ and ${\bf \Delta}u(\phi)=0$ for any $\phi\in {\rm Test}^\infty_c(X)$. Since ${\rm Test}^\infty_c(X)$ is dense in $W^{1,2}_0(X)$ (see Lemma \ref{lem-add1}), we conclude that $u$ is a harmonic function on $X$.
		\end{proof}

Now we are in the position to show Corollary \ref{cor1.6}
	\begin{proof} [Proof of Corollary \ref{cor1.6}]
	
		Let $u$ be a $L^1$  very weakly subharmonic function on $X$. From Proposition \ref{prop5.2}, we know in fact that $u$ is  in $  Lip_{\rm loc}(X)$ and is a harmonic function on $X$.

We consider the function $u_1:=\min\{u,a\}$ for arbitrary $a>0$.  Since $u$ is harmonic, according to \cite[Lemma 2]{Stu94}, we have that $u_1$ is superharmonic. In particular, it is very weakly superharmonic on $X$ (by Lemma \ref{lem3.3}). Noticing that $u_1\in L^1(X)$,  by using Proposition \ref{prop5.2} again,  we get that $u_1$ is again harmonic on $X$. We further put $u_2:=\max\{u_1,-a\}$. Similar, from the harmonicity of $u_1$,  \cite[Lemma 2]{Stu94} implies that $u_2$ is subharmonic on $X$; and then again,  Proposition \ref{prop5.2} yields  that $u_2$ is  also harmonic on $X$. From \cite[Lemma 2]{Stu94}, we know that $|u_2|$ is a nonnegative subharmonic function on $X$. Since $-a\leqslant u_2\leqslant a$, we get that $u_2\in L^\infty(X)\cap L^1(X)\in L^2(X)$.
According to \cite[Theorem 1]{Stu94}, we now conclude that $|u_2|$ is a  constant. It follows from the arbitrariness of $a$ that $u$ is a constant. The proof of Corollary \ref{cor1.6} is finished. 
\end{proof}

\section{A Remark on gradient estimates for solutions  of elliptic equations}\label{sec6}
In this section,  we will show Theorem \ref{theorem-1.8}. 

We first deal with the harmonic functions on $RCD(0,N)$ cones.  We recall
that if $N\geqslant 2$ and if  $(\Sigma, d_\Sigma, \meas_{\Sigma})$ is a metric measure space, then the cone $Cone^N_0(\Sigma):= (C, d_c, \meas_c)$ is
the metric measure space given by the warped product metric $\Sigma \times_{r^{N-1}}\mathbb R$ and $d\meas_c(r,\xi):=r^{N-1}drd\meas_\Sigma(\xi)$.  
It was proved \cite{Ket15a}  that  $Cone^N_0(\Sigma)$ is an $RCD(0,N)$ space if and only    if $(\Sigma,d_\Sigma,\meas_\Sigma)$ is an $RCD(N-2,N-1)$ space.
 
\begin{lemma}\label{lem6.1}
Let $(C,d_c,\meas_c)$ be an $RCD(0,N)$ cone with vertex $o$, over $(\Sigma, d_\Sigma,\meas_\Sigma)$.   Suppose that  $u\in W^{1,2}( B_R(o) )$ is a (weakly) harmonic function, then 
\begin{equation}\label{equation-6.1}
\fint_{B_{r_1}(o)}|\nabla u|^2d\meas_c\leqslant \fint_{B_{r_2}(o)}|\nabla u|^2d\meas_c
\end{equation} 
for all $0<r_1<r_2<R$. 
 \end{lemma}
\begin{proof}
According to \cite[Theorem 3.1]{Huang20}, there exist $\{c_i\}_{i}\subset \mathbb R$  such that
$$u(r,\xi)=\sum_{i=0}^\infty c_i r^{\alpha_i}\phi_i(\xi),$$
where $\phi_i(\xi)$ is an eigenfunction of $-\Delta_\Sigma$ with respect to the eigenvalue $\lambda_i$, and $\alpha_i>0$ are given by $\lambda_i=\alpha_i(N+\alpha_i-2)$. 
From \cite{Ket15b,JZ16}, we have the first eigenvalue $\lambda_1\geqslant N-1$. It follows $\alpha_i\geqslant 1$ for all $i\geqslant 1.$
A direct calculation (see \cite[Eq. (2.5)]{DPNZ22}) implies that 
$$\frac{1}{r^N}\int_{B_r(o)}|\nabla u|^2d\meas_c=\sum_{i=1}^\infty c^2_i \alpha_i \cdot r^{2\alpha_i-2}.$$
 The desired monotonicity (\ref{equation-6.1}) follows from this and the fact $\meas_c(B_r(o))=\frac{r^N}{N}\cdot \meas_{\Sigma}(\Sigma).$
\end{proof}

Let $n\geqslant 3$. Let $u\in W^{1,2}_{\rm loc}(B_1(0))$ be a weak solution of elliptic equations of divergence form ${\rm div}(A(x)Du)=0,$
where the matrix of coefficients $A(x)=( a^{ij})_{i,j=1}^n$ is  symmetric, bounded measurable and satisfies (\ref{equation-1.6}).

\begin{proof}[Proof of Theorem \ref{theorem-1.8}]
According to the assumption of   $A$ in Theorem \ref{theorem-1.8},
   there exists a conical coefficient of nonnegative curvature $\overline{A}=(\overline{a}^{ij})$ such that (\ref{equ-dini}) holds.
 Since $\omega_{A,\overline A}(r)\to 0$ as $r\to0$, there is  $r_1>0$ such that
\begin{equation*} 
\|\overline{a}^{ij}\|_{L^\infty(B_{r_1}(0))}\leqslant 2  \Lambda,\qquad   \frac{ \lambda}{2}|\xi|^2\leqslant \sum_{i,j=1}^n\overline{a}^{ij}(x)\xi_i\xi_j,\quad {\rm a.e.} \ x\in B_{r_1}(0), \quad \forall \xi\in \mathbb R^n.
\end{equation*}
It is easy to check that the the  distance $d_{\overline{g}}$ induced from the Riemannian metric $\overline{g}_{ij}$ in \eqref{equation-1.7} is bi-Lipschitz equivalent to the Euclidean metric on $B_{r_1}(0)$. The bi-Lipschitz constant depends only on $ \lambda$ and $ \Lambda$.

 We denote by $B^{\overline{g}}_r(0)$ the ball with respect the metric $d_{\overline g}$, for any $r>0$.  
 Then there exists a constant $C_1=C_{1}( \lambda,  \Lambda)>1$ such that 
\begin{equation}\label{equation-6.2}
B_{r/C_1}(0)\subset B_r^{\overline{g}}(0)\subset B_{C_1r}(0),\qquad \forall r>0.
\end{equation}
The bi-Lipschitz equivalence between $d_{\overline g}$ and the Euclidean distance implies that 
$$ W^{1,2}(B_{r}^{\overline g}(0), d_{\overline g}, {\rm vol}_{\overline g})=W^{1,2}(B_r(0))$$
 for any $r\in (0,r_1/C_1)$, where the Riemannian measure 
\begin{equation}\label{equation-6.3}
{\rm vol}_{\overline g}=\sqrt{G}\cdot \mathcal L^n,\qquad G:= {\rm det}(\overline g_{ij}),
\end{equation}
 and $\mathcal L^n$ is the Lebesgue measure on $\mathbb R^n$.
From (\ref{equation-1.7}), we have 
\begin{equation}\label{equation-6.4}
G=\left({\rm det}{(\overline  a^{ij})}\right)^{\frac{2}{n-2}},\qquad \overline{a}^{ij}=  \overline{g}^{ij}\cdot\sqrt G
\end{equation} for all $1\leqslant i,j\leqslant n.$

Denote $B^{\overline g}_\ell:=B^{\overline g}_{\rho^\ell}(0),$ where $\rho:=1/C_1<1$. Let $\ell_0\in\mathbb N$, $\ell_0>2$, such that $\rho^{\ell_0}<r_1/C_1$. For each $\ell=\ell_0,\ell_0+1, \cdots, $ let $u_\ell$ be the solution 
of Dirichlet problem
$$\partial_i\left(\overline a^{ij}\partial u_\ell\right)=0 \quad{\rm with}\quad u-u_\ell\in W^{1,2}_0(B^{\overline g}_\ell)\ \  \left(=W^{1,2}_0(B^{\overline g}_\ell, d_{\overline g}, {\rm vol}_{\overline g})\right).$$
Noticing that   $\partial_i(a^{ij}\partial_ju)=0$ in the weak sense and  letting $v_\ell=u-u_\ell$, we have 
\begin{equation*}
\int_{B^{\overline g}_\ell}\overline a^{ij}\partial_j v_\ell\partial_iv_\ell d\mathcal L^n=\int_{B^{\overline g}_\ell}\overline a^{ij}\partial_j u\partial_iv_\ell d\mathcal L^n=\int_{B^{\overline g}_\ell}(\overline a^{ij}-a^{ij})\partial_j u\partial_iv_\ell d\mathcal L^n.
\end{equation*}
Together this,  (\ref{equation-6.2}), (\ref{equation-6.3}), (\ref{equation-6.4}) and (\ref{equ-dini}), we get
\begin{equation*}
\begin{split}
\|\nabla_{\overline g}v_\ell\|^2_{L^2(B_\ell^{\overline g},{\rm vol}_{\overline g})} &=\int_{B^{\overline g}_\ell}\overline g^{ij}\partial_j v_\ell\partial_iv_\ell   d{\rm vol}_{\overline g}=\int_{B^{\overline g}_\ell}\overline a^{ij}\partial_j v_\ell\partial_iv_\ell d\mathcal L^n\\
&\leqslant C_2\cdot\omega_{A,\overline A}(C_1\rho^\ell)\cdot\|\nabla_{\overline g}u\|_{L^2(B_\ell^{\overline g},{\rm vol}_{\overline g})} \cdot  \|\nabla_{\overline g}v_\ell\|_{L^2(B_\ell^{\overline g},{\rm vol}_{\overline g})}, 
\end{split}
\end{equation*}
where the constant $C_2$ depends only on $\lambda,\Lambda$, and we have used that $\|\partial_ju\|_{L^2(B_\ell^{\overline g},\mathcal L^n)}\leqslant C_{\lambda,\Lambda} \|\nabla_{\overline g}u\|_{L^2(B_\ell^{\overline g},{\rm vol}_{\overline g})}$ for some constant $C_{\lambda,\Lambda}>0$. Hence,   for each $\ell=\ell_0,\ell_0+1,\cdots,$ we have (noticing that $C_1=\rho^{-1}$)
\begin{equation}\label{equation-6.5}
\|\nabla_{\overline g}v_\ell\|_{L^2(B_\ell^{\overline g},{\rm vol}_{\overline g})} \leqslant C_2\cdot\omega_{A,\overline A}(\rho^{\ell-1})\cdot\|\nabla_{\overline g}u\|_{L^2(B_\ell^{\overline g},{\rm vol}_{\overline g})}.
\end{equation}

On the other hand, we claim that, for each $\ell=\ell_0,\ell_0+1,\cdots,$   $u_\ell$ is a (weakly) harmonic function on $B^{\overline g}_\ell$ in metric measure space $(\mathbb R^n, d_{\overline g}, {\rm vol}_{\overline g})$. Indeed, for any $\phi\in Lip_0(B^{\overline g}_\ell)$, we have
$$\int_{B^{\overline g}_\ell}\ip{ \nabla_{\overline g}u_\ell}{\nabla_{\overline g}\phi}d{\rm vol}_{\overline g} =\int_{B^{\overline g}_\ell}\overline a^{ij}\partial_j u_\ell\partial_i\phi d\mathcal L^n=0.$$

Since  $(\mathbb R^n, d_{\overline g})$ has nonnegative curvature in the sense of Alexandrov, the metric measure space $(\mathbb R^n,  d_{\overline g}, {\rm vol}_{\overline g})$ is an $RCD(0,n)$ cone (see \cite{Pet11,ZZ10}). From Lemma \ref{lem6.1},  we have
$$ \fint_{B_{\ell+1}^{\overline g}(0)}|\nabla_{\overline g}u_\ell|^2d{\rm vol}_{\overline g} \leqslant  \fint_{B_{\ell}^{\overline g}(0)}|\nabla_{\overline g}u_\ell|^2d{\rm vol}_{\overline g}$$
for each $\ell=\ell_0,\ell_0+1,\cdots.$ By combining with (\ref{equation-6.5}) and $u=u_\ell+v_\ell$ on $B^{\overline g}_\ell$, we get
{\small \begin{equation*}
\begin{split}
& \left(\fint_{B_{\ell+1}^{\overline g}(0)}|\nabla_{\overline g}u|^2d{\rm vol}_{\overline g} \right)^{1/2} \\
\leqslant &  \left(\fint_{B_{\ell+1}^{\overline g}(0)}|\nabla_{\overline g}u_\ell|^2d{\rm vol}_{\overline g} \right)^{1/2}+ \left(\fint_{B_{\ell+1}^{\overline g}(0)}|\nabla_{\overline g}v_\ell|^2d{\rm vol}_{\overline g} \right)^{1/2}  \\
 \leqslant& \left(\fint_{B_{\ell}^{\overline g}(0)}|\nabla_{\overline g}u_\ell|^2d{\rm vol}_{\overline g} \right)^{1/2}+ \frac{{\rm vol}_{\overline g}(B^{\overline g}_\ell)}{{\rm vol}_{\overline g}(B^{\overline g}_{\ell+1})}\cdot \left(\fint_{B_{\ell}^{\overline g}(0)}|\nabla_{\overline g}v_\ell|^2d{\rm vol}_{\overline g} \right)^{1/2}  \\
 \leqslant &\left(\fint_{B_{\ell}^{\overline g}(0)}|\nabla_{\overline g}u|^2d{\rm vol}_{\overline g} \right)^{1/2}+ \frac{  C_2}{\rho^n}\cdot \omega_{A,\overline A}(\rho^{\ell-1}) \left(\fint_{B_{\ell}^{\overline g}(0)}|\nabla_{\overline g}u|^2d{\rm vol}_{\overline g} \right)^{1/2}  \\
 \leqslant & \left(1+C_3\omega_{A,\overline A}(\rho^{\ell-1})\right)  \left(\fint_{B_{\ell}^{\overline g}(0)}|\nabla_{\overline g}u|^2d{\rm vol}_{\overline g} \right)^{1/2},   \end{split}
 \end{equation*}}
for any $\ell=\ell_0,\ell_0+1,\cdots,$ where we have used the $\|\nabla_{\overline g}u_\ell\|_{L^2(B_\ell^{\overline g},{\rm vol}_{\overline g})} \leqslant \|\nabla_{\overline g}u\|_{L^2(B_\ell^{\overline g},{\rm vol}_{\overline g})} $ by the property that $u_\ell$ is the  energy minimizing function, and the doubling property of ${\rm vol}_{\overline g}$ (In fact $\frac{{\rm vol}_{\overline g}(B^{\overline g}_\ell)}{{\rm vol}_{\overline g}(B^{\overline g}_{\ell+1})}=\rho^{-n}$).
  By the elementary inequality $\ln(1+s)\leqslant s$ for $s>0$, we have
{\small \begin{equation*}
 \ln \left(\fint_{B_{\ell+1}^{\overline g}(0)}|\nabla_{\overline g}u|^2d{\rm vol}_{\overline g} \right)  \leqslant  \ln \left(\fint_{B_{\ell}^{\overline g}(0)}|\nabla_{\overline g}u|^2d{\rm vol}_{\overline g} \right)+2C_3\omega_{A,\overline A}(\rho^{\ell-1}) 
  \end{equation*}}
for any $\ell=\ell_0.\ell_0+1,\cdots.$ Therefore, we have
{\small \begin{equation*}
\limsup_{\ell\to+\infty} \ln \left(\fint_{B_{\ell+1}^{\overline g}(0)}|\nabla_{\overline g}u|^2d{\rm vol}_{\overline g} \right)  \leqslant  \ln \left(\fint_{B_{\ell_0}^{\overline g}(0)}|\nabla_{\overline g}u|^2d{\rm vol}_{\overline g} \right)+2C_3\sum_{\ell=\ell_0}^\infty\omega_{A,\overline A}(\rho^{\ell-1}).
  \end{equation*}}
Since $\omega_{A,\overline A}(t)$ is nondecreasing in $t>0$, we have
$$\sum_{\ell=\ell_0-1}^\infty\omega_{A,\overline A}(\rho^\ell)\leqslant \sum_{\ell=\ell_0-1}^\infty\int_{\rho^\ell}^{\rho^{\ell-1}}\frac{\omega_{A,\overline A}(t)}{\rho^{\ell-1}-\rho^\ell}\cdot \frac{\rho^{\ell-1}}{t}dt= \frac{1}{1-\rho}\int_{0}^{\rho^{\ell_0-2}}\frac{\omega_{A,\overline A}(t)}{t}dt.$$
Therefore, we conclude that 
{\small \begin{equation*}
\limsup_{\ell\to+\infty}  \left(\fint_{B_{\ell+1}^{\overline g}(0)}|\nabla_{\overline g}u|^2d{\rm vol}_{\overline g} \right)  \leqslant   \left(\fint_{B_{\ell_0}^{\overline g}(0)}|\nabla_{\overline g}u|^2d{\rm vol}_{\overline g} \right)\cdot \exp\left(\frac{2C_3}{1-\rho}\int_0^1  \frac{\omega_{A,\overline A}(t)}{t}dt\right).
  \end{equation*}
 }It follows the desired estimate (\ref{equation-1.9}), by  combining the doubling property of ${\rm vol}_{\overline g}$ and (\ref{equation-6.2})-(\ref{equation-6.4}). The proof is finished. 
  \end{proof}

\end{document}